\documentclass[12pt, oneside]{amsart}       

\usepackage[T1]{fontenc}
\usepackage{palatino}

\usepackage{graphicx}                
                            \usepackage{amsthm}    
\newtheorem{theorem}{Theorem}[section]

\newtheorem{lemma}[theorem]{Lemma}
\theoremstyle{remark}
\newtheorem*{remark}{Remark}
\usepackage{tcolorbox}
\newtheorem*{example}{Example}

\title{Infinitesimally Equivariant Bundles on Complex Manifolds}

\author{Emile Bouaziz}

\begin{document} \maketitle

\begin{abstract} We show that any continuous $\mathbf{C}$-linear Lie algebra splitting of the symbol map from the Atiyah algebra of a vector bundle on a complex manifold is given by a differential operator of order at most the rank of the bundle plus one. Bundles equipped with such a splitting can be thought of as \emph{infinitesimally equivariant} bundles, and our theorem implies these are, in a certain sense, in a categorical formal neighbourhood of vector bundles with a flat connection.   \end{abstract}

\section{Introduction}

\subsection{Basic Notions and Infinitesimal Equivariance.} Let $X$ be a complex manifold, with sheaf of holomorphic functions $\mathcal{O}_{X}$ and sheaf of holomorphic vector fields $\Theta_{X}$. If $\mathcal{V}$ is a holomorphic bundle of finite rank on $X$, we have the \emph{Atiyah algebra}, $\operatorname{At}_{X}(\mathcal{V})$, the sheaf of first order differential operators from $\mathcal{V}$ to itself with symbol in $\Theta_{X}\subset\Theta_{X}\otimes\operatorname{End}(\mathcal{O}_{X})$. Local sections are pairs $(\eta,\widetilde{\eta})$ where $\eta$ is a vector field and $\widetilde{\eta}$ is an endomorphism of $\mathcal{V}$ satisfying the Leibniz rule $\widetilde{\eta}(fs)=f\widetilde{\eta}s+\eta(f)s$. $\operatorname{At}_{X}(\mathcal{V})$ is a sheaf of $\mathbf{C}$-Lie algebras and the \emph{symbol} map, $\sigma$, to $\Theta_{X}$,  is a Lie map. The reader is referred to \cite{BS} for an excellent discussion of Atiyah algebras.

 A choice of continuous (cf. the remark preceding Theorem 1.1.) $\mathbf{C}$-Lie algebra splitting of $\sigma$ can be thought of as an action of infinitesimal symmetries on $\mathcal{V}$. Informally, this expresses a certain naturality of $\mathcal{V}$. In particular, one expects such a structure on any suitably natural sheaf. Below we will denote by $L$ a given splitting of $\sigma$, and refer to $L$ as the \emph{Lie map}. Such a choice of $L$ will be referred to as an \emph{infinitesimally equivariant} structure on $\mathcal{V}$, or an \emph{infeq} structure for short. The category of such is denoted $\mathbf{InfEq}(X)$, and its elements are referred to as \emph{infeq bundles}. Given an infeq bundle $\mathcal{V}$ and a vector field $\eta$, the corresponding endomorphism of $\mathcal{V}$ will be denoted $L_{\eta}$, and referred to as the \emph{Lie derivative} by $\eta$.

\begin{example} \begin{itemize}\item An infeq structure on $\mathcal{V}$ for which $L$ is $\mathcal{O}$-linear is the data of a flat connection on $\mathcal{V}$, as can easily be checked. 
\item  The sheaves of forms $\Omega^{i}_{X}$ are naturally infeq bundles, with the usual Lie derivative action. Note that in this case the map $L$ is not $\mathcal{O}$-linear. In fact it is a first order differential operator - this follows from the Cartan formula expressing $L_{\eta}$ as the commutator of the $\mathcal{O}$-linear contraction, $\iota_{\eta}$ with the first order differential operator $d_{dR}$.
\item The sheaf, $D^{\leq n}_{X}$, of differential operators of order at most $n$ is naturally an infeq bundle on $X$. In this case $L$ has order $n$ as a differential operator.
\item There are plenty of examples of bundles admitting no infeq structure, indeed one expects a generic bundle not to admit one. The simplest example is a line bundle of non-zero degree on an elliptic curve, cf. subsection 3.3. below for a sketch of a proof. \end{itemize}\end{example}

\begin{remark} The sheaf $\mathcal{O}$ is naturally a sheaf of topological vector spaces, where the topology is induced by the topology of uniform convergence on compacta, so we can make sense of continuous morphisms between trivial bundles.  A morphism between bundles is called continuous if it is so with respect to a trivialising cover of $X$. This is independent of choices of trivialisation. Differential operators, and in particular $\mathcal{O}$-linear morphisms, are continuous. \end{remark}

The examples given above are all such that $L$ is a differential operator of some order. This is in fact forced, and is our main theorem. 
\begin{theorem} Let $\mathcal{V}$ be an infinitesimally equivariant bundle on $X$, with Lie map $L$. Then $L$ is a differential operator of order at most $\operatorname{rank}(\mathcal{V})+1.$  If $\mathcal{V}$ is a line bundle, then in fact $L$ has order at most $1$. \end{theorem}

\section{acknowledgements} We have benefited from numerous conversations with Yuly Billig, Colin Ingalls and Henrique Rocha during the writing of this note. An algebraic version of some of the results presented here is the subject of  joint work of the author with Rocha, cf. \cite{BR}.

\section{Proof of Main Theorem}

\subsection{Infinite Order Differential Operators} Our proof of theorem 1.1 relies on a beautiful result of Ishimura (\cite{Ish}). Ishimura proves that continuous endomorphisms of $\mathcal{O}_{X}$ are (uniquely) represented by differential operators \emph{of possibly infinite order}. The (formally defined) symbols of these differential operators are required to satisfy certain growth conditions. For example, it is a pleasant calculation with the Cauchy integral formula to see that if the $\lambda_{i}$ are constants, then $\sum_{n}\lambda_{n}\partial^{n}$ acts on functions on $\mathbf{C}$ in a local way iff $\sum_{n}\lambda_{n}n!z^{n}$ is entire. Nonetheless, these growth conditions will not concern us, as we need only the representation as an infinite order differential operator.
 
Let us fix now some notation, $\Delta^{d}$ will denote a complex polydisc of dimension $d$, with coordinates $z_{i}, \,i=1,...,d$. We write $z=(z_{1},...,z_{d})$. The derivations $\partial_{j}$ satisfying $\partial_{j}(z_{i})=\delta_{ij}$ give a trivialisation of $\Theta_{\Delta^{d}}$. As usual, if $I=(i_{1},i_{2},...,i_{d})\in\mathbf{Z}^{d}_{+}$ is a multi-index, we write $z^{I}:=\prod_{j}z_{j}^{i_{j}}$, $\partial^{I}:=\prod_{j}\partial_{j}^{i_{j}}$. Further we write $\operatorname{wt}(I):=\sum_{j}i_{j}$ and refer to it as the \emph{weight} of the multi-index. Finally, we write $I!:=\prod_{j}i_{j}!$. Infinite order differential operators on $\Delta^{d}$ are thus represented as infinite sums, $$\sum_{I}a_{I}(z)\partial^{I}, \, a_{I}(z)\in\mathcal{O}(\Delta^{d}).$$ 

\subsection{Proof of main theorem} We are now in a position to prove the main theorem of this note. We record first the following lemma, which is purely Lie theoretical, certainly well known, and at any rate quite simple. Note that the importance of the representation theory of Lie algebras of vector fields on discs to the study of infeq objects is clear from \cite{BFN} and \cite{BIN}, who work in an algebaic context. The relevance of the representation theory of $\mathfrak{g}_{d}$ should also be clear to the reader familiar with the Gelfand-Kazhdan \emph{formal geometry}, cf. \cite{GGW}.

\begin{lemma} Let $\mathfrak{g}_{d}$ be the topological Lie algebra of holomorphic vector fields on $\Delta^{d}$ which vanish at the origin. For each $N$, let $\mathfrak{g}_{d}^{N}$ be the quotient of $\mathfrak{g}_{d}$ by the ideal of vector fields on $\Delta^{d}$ vanishing to order $N+2$ at the origin. Then if $\mathfrak{g}_{d}$ acts continuously on a $\mathbf{C}$ vector space $V$ of dimension $r$, the action factors through the quotient $\mathfrak{g}_{d}^{r}$. Further, if $r=1$, the quotient actually factors through $\mathfrak{g}_{d}^{0}$. \end{lemma} 

\begin{proof} There is a natural circle action on $\Delta^{d}$, and $\mathfrak{g}_{d}$ is topologically spanned by weight vectors for this action. Further, all weights occuring are non-negative integers and $\mathfrak{g}_{d}^{N}$ is the quotient by the ideal spanned (topologically) by vectors of weight at least $N+1$. Finally, if $\nu:=\sum_{i}z_{i}\partial_{i}$ is the Euler vector field, and $\eta$ is of weight $w$, then $[\nu,\eta]=w\eta$. 

We first show that the action of $\mathfrak{g}_{d}$ factors through the quotient $\mathfrak{g}_{d}^{N}$ for large enough $N$. Call the representation $\rho$. By continuity it suffices to show that there is some $N$ such that any vector of weight at least $N+1$ acts as $0$. If the image of the Euler vector field, $\rho(\nu)$ vanishes, then this holds trivially, as any non-zero weight vector is in the ideal generated by $\nu$. We may thus assume that $\nu$ maps to a non-zero element of $\mathfrak{gl}(V)$. Such an element can have only finitely many distinct eigenvectors when acting on $\mathfrak{gl}(V)$ via the adjoint representation. It follows immediately that there exists an $N$ as claimed, as the image of any weight vector of weight $w$ is an eigenvector for $\rho(\nu)$ acting on $\mathfrak{gl}(V)^{\operatorname{ad}}$. 

Now we show that we can take $N=r$. The Lie algebra $\mathfrak{g}^{N}_{d}$ is a finite dimensional solvable Lie algebra, and so we know that the image of $\rho$ is contained in a Borel subalgebra, by Lie's theorem. We must now show that any vector of weight at least $N+1$ is in the $N$-th derived subalgebra of $\mathfrak{g}_{d}^{N}$. This can be proven by an easy induction, using the operators $[z_{i}^{2}\partial_{i},-]$.

Finally, when $r=1$, the action must factor through the abelianization of $\mathfrak{g}_{d}$, which is easily seen to be a quotient of $\mathfrak{g}_{d}^{0}$, as non-zero weight vectors lie in the image of $[\nu,-]$.
\end{proof}

\begin{remark} Note that the above implies there is nothing \emph{holomorphic} about the category of continuous finite dimensional representations of $\mathfrak{g}_{d}$. As the representations factor through a quotient defined by tangency conditions at $0$, the vector fields may as well be formal algebraic such. \end{remark}

\begin{theorem} Let $\mathcal{V}$ be an infinitesimally equivariant bundle on $X$, with Lie map $L$. Then $L$ is a differential operator of order at most $\operatorname{rank}(\mathcal{V})+1.$  If $\mathcal{V}$ is a line bundle, then in fact $L$ has order at most $1$. \end{theorem}

\begin{proof} It suffices to prove the result when $X$ is a polydisc, $\Delta^{d}$, and $\mathcal{V}$ is the trivial bundle of rank $r$, $\mathcal{O}^{\oplus r}$. The Atiyah algebra, $\operatorname{At}_{\Delta^{d}}(\mathcal{O}^{\oplus r})$, is isomorphic to the semi-direct product of $\Theta:=\Theta_{\Delta^{d}}$ with $\mathfrak{gl}_{r}(\mathcal{O})$, where $\Theta$ acts on $\mathfrak{gl}_{r}(\mathcal{O})$ in the natural manner, denoted $*$ below. With these simplifications we see that the Lie map, $L$, corresponds to a continuous map of sheaves, $\widetilde{L}:\Theta\rightarrow\mathfrak{gl}_{r}(\mathcal{O})$, satisfying the following \emph{(non-abelian) cocycle} identity for all pairs of vector fields;

$$\widetilde{L}\big([\eta_{0},\eta_{1}]\big)=\eta_{0}*\widetilde{L}(\eta_{1})-\eta_{1}*\widetilde{L}(\eta_{0})+\big[\widetilde{L}(\eta_{0}),\widetilde{L}(\eta_{1})\big].$$

Now, by Ishimura's theorem, we know that we can write $\widetilde{L}$ as an infinite order differential operator from $\Theta$ to $\mathfrak{gl}_{r}(\mathcal{O})$. We see then that there are matrices $A^{i}_{I}\in\mathfrak{gl}_{r}(\mathcal{O})$, depending on $i\in\{1,2,...,d\}$ and $I\in\mathbf{Z}_{+}^{d}$, so that $\widetilde{L}$ is given by the matrix differential operator $$\sum_{i,I}A^{i}_{I}dz_{i}\partial^{I}.$$

By the cocycle identity applied to $\eta_{0}=\partial_{i}$ and $\eta_{1}=\partial_{j}$, we deduce immediately that $$\partial_{i}A^{j}_{\bf{0}}-\partial_{j}A^{i}_{\mathbf{0}}=\big[A^{i}_{\bf{0}},A^{j}_{\bf{0}}\big].$$ Equivalently, the matrix valued one form $A_{\bf{0}}:=\sum_{i}A^{i}_{\bf{0}}dz_{i}$ defines a flat connection on $\mathcal{O}^{\oplus r}$, denoted $\nabla_{A_{\bf{0}}}$. We now let $\eta_{1}$ be arbitrary, and fix $\eta_{0}=\partial_{i}$. We obtain the relation in the space of continuous homomorphisms of sheaves from $\Theta$ to $\mathfrak{gl}_{r}(\mathcal{O})$; $$\big[\widetilde{L},\partial_{i}\big]=\widetilde{L}(\partial_{i})dz_{i}+\operatorname{adj}_{\widetilde{L}(\partial_{i})}(\widetilde{L}),$$ where $\operatorname{adj}$ refers to the natural adjoint action of $\mathfrak{gl}_{r}(\mathcal{O})$ on the space of such homomorphisms. Substituting $\sum_{i,I}A^{i}_{I}dz_{i}\partial^{I}$ for $\widetilde{L}$ and equating coefficients of $\partial^{J}$ with $\operatorname{wt}(J)> 0$, we deduce the following differential equations; $$\partial_{i}A^{j}_{J}=\big[A^{i}_{\bf{0}},A^{j}_{J}\big].$$ Equivalently, we see that $A^{j}_{J}$, for $\operatorname{wt}(J)>0$ are flat sections of the adjoint connection, $\nabla^{\operatorname{adj}}_{A_{\bf{0}}}$. In particular, these matrices vanish as soon as their value at $0\in\Delta^{d}$ does.

Now, we recall the Lie algebra $\mathfrak{g}_{d}$, which naturally lies inside $\Gamma(\Delta^{d},\Theta)$. We can restrict $\widetilde{L}$ to $\mathfrak{g}_{d}$ and evaluate at $0\in\Delta^{d}$ to produce a linear map, $\rho$, from $\mathfrak{g}_{d}$ to $\mathfrak{gl}_{r}(\mathbf{C})$. The terms $\eta_{0}*\widetilde{L}(\eta_{1})$ and $\eta_{1}*\widetilde{L}(\eta_{0})$ both vanish when evaluated at $0$ (by the definition of $\mathfrak{g}_{d}$), hence the resulting map is a morphism of Lie algebras. Applied to a vector field $z^{J}\partial_{j}$, $\rho$ simply produces $J!A^{j}_{J}(0)$, which we deduce vanishes for $\operatorname{wt}(J)\geq r+2$ by lemma 2.1. As remarked above, this implies that the matrices $A^{j}_{J}$ vanish for $\operatorname{wt}(J)\geq r+2$, as they are flat sections of the adjoint connection $\nabla^{\operatorname{adj}}_{A_{\bf{0}}}$, whence $L$ is a differential operator of order at most $r+1$. Further, if $r=1$ it is clear that we obtain a differential operator of order at most $1$, and the theorem is proven.\end{proof}

\begin{remark} We regret that we do not know if the above bound $\operatorname{rank}(\mathcal{V})+1$ is optimal for every choice of $d$ and $r$. Of course, the proof above makes it clear that this is a purely Lie theoretic question. \begin{itemize}\item For $r=1$ it the bound of $1$ is always obtained - for example by the determinant bundle $\omega_{X}$. \item For $d=1$ we can always find a rank $r$ infeq bundle so that $L$ has order $r$, indeed we can take the sheaf of $r$-jets of sections of $\mathcal{O}$. \item For $d=1$ and $r=2$ we can achieve the bound of $3$ with the infeq bundle on $\Delta^{1}$ whose Lie map is $h\partial+e\partial^{3}$, with $h,e\in\mathfrak{sl}_{2}(\mathbf{C})\subset\mathfrak{gl}_{2}(\mathcal{O})$ the evident elements.  \end{itemize}\end{remark}

\subsection{Local Structure, Obstructions and Further Examples.}

The proof of theorem 2.2 essentially proves that there is an equivalence of categories between $\mathbf{InfEq}(\Delta^{d})$ and $\mathbf{Rep}^{\operatorname{fin}}(\mathfrak{g}_{d})$. Indeed the connection $\nabla_{A_{\bf{0}}}$ is trivialisable, and as the matrices $A^{j}_{J}$ with $\operatorname{wt}(J)>0$ are flat sections of $\nabla_{A_{\bf{0}}}$, a suitable change of frame makes all the $A^{j}_{J}$ constant. It is then an easy exercise to see that infeq structures on a trivial bundle for which the Lie map is a constant coefficient differential operator correspond exactly to finite dimensional representations of $\mathfrak{g}_{d}$. Reiterating the above, we have the following equivalence. \begin{theorem} There is an equivalence of categories $$\mathbf{InfEq}(X)\cong\mathbf{Rep}^{\operatorname{fin}}(\mathfrak{g}_{d}).$$ \end{theorem}\begin{proof} Follows from the argument of theorem 2.2.\end{proof} We stress an important consequence of the above. If $\mathcal{V}$ is an infeq bundle on a general (connected) $X$, then there is a well defined (up to isomorphism) $\rho(\mathcal{V})\in\mathbf{Rep}^{\operatorname{fin}}(\mathfrak{g}_{d})$ which describes the local structure of $\mathcal{V}$, i.e.the isomorphism type, with respect to the isomorphism of theorem 3.3, of the restriction of $\mathcal{V}$ to a disc in $X$. In particular for $\rho$ a fixed representation of $\mathfrak{g}_{d}$ on a finite dimensional space, we obtain a full subcategory, $\mathbf{InfEq}^{\rho}(X)$, of $\mathbf{InfEq}(X)$. We think of the complexity of the representation $\rho$ as measuring how far $\mathcal{V}$ is from being a vector bundle with flat connection, indeed $\rho(\mathcal{V})$ is zero iff $\mathcal{V}$ is a vector bundle with flat connection.

\begin{example}\begin{itemize}\item For fixed dimension $r$ and $\rho=0$, we obtain the category of vector bundles of rank $r$ with a flat connection. \item $\Omega_{X}^{1}$ is naturally an element of $\mathbf{InfEq}(X)$. It is easy to see that $\rho_{\Omega^{1}_{X}}$ is the pull-back to $\mathfrak{g}_{d}$ of the standard representation of $\mathfrak{gl}_{d}\cong\mathfrak{g}_{d}^{1}$. \item We know fix $d=r=1$, so let $X$ be a smooth projective curve and let $\mathcal{L}$ be a line bundle on $X$. Let $\rho$ be a fixed $1$-dimensional representation of $\mathfrak{g}_{1}$. $\rho$ is equivalent to the data of a complex number, which we also denote $\rho$. On each complex disc in $X$, there is a well defined infeq sheaf corresponding to $\rho$. The obstruction to a global such with underlying line bundle $\mathcal{L}$ is $c_{1}(\mathcal{L})-\rho c_{1}(X)$. This is a relatively straightforward generalisation of the case $\rho=0$, which is a well known theorem. In this case, an infeq structure is a flat connection. Covering $X$ by discs $U_{i}$, a flat connection on $\mathcal{L}$ is equivalent to one forms $\lambda_{i}$ on $U_{i}$ so that the we have \emph{gauge conditions}, $$\lambda_{i}-\lambda_{j}=\operatorname{dlog}(\phi_{ij}),$$ where $\phi_{ij}$ are the frame transformations of $\mathcal{L}$ with respect to our cover. The $\lambda_{i}$ are thus a one-cochain with boundary the cocycle $\operatorname{dlog}(\phi_{ij})$, which is well known to represent $c_{1}(\mathcal{L})$. The case of general $\rho$ follows by similar cocycle manipulations, as we know that the Lie map $L$ is a differential operator of order at most $1$. In particular we note that line bundles of non-zero degree on an elliptic curve have no infeq structure, and any other line bundle on a smooth projective curve admits at least one such.\end{itemize}\end{example}


\begin{thebibliography}{9}
\bibitem{BS}
A.A. Beilinson, V.V. Schechtman,
\textit{Determinant Bundles and Virasoro Algebras.}

Communications in Mathematical Physics, volume 118, pages 651-701 (1988)

\bibitem{BFN}
Y. Billig, V. Futorny, J. Nilsson,
\textit{Representations of Lie algebras of vector fields on affine varieties.}
Israel Journal of Mathematics, volume 233, pages 379-399 (2019).

\bibitem{BIN}
Y. Billig, C. Ingalls, A. Nasr,
 \textit{AV modules of finite type on affine space.}
arXiv:2002.08388

\bibitem{BR}
E. Bouaziz, H. Rocha,
\textit{Annihilators of AV modules and differential operators.}
arXiv:2211.09211

\bibitem{GGW}
V. Gorbounov, O. Gwilliam, B. Williams,
\textit{Chiral differential operators via Batalin-Vilkovisky quantization}
Asterisque, 419. (2020)

\bibitem{Ish}
R. Ishimura,
\textit{Homomorphismes du faisceau des germes de fonctions holomorphes dans lui-meme et operateurs differentiels.}

Memoirs of The Faculty of Science, Kyushu University, 1978.
























\end{thebibliography}
\end{document}